\documentclass[12pt]{amsart}

\usepackage{epsfig}
\usepackage{amsthm,amsfonts}
\usepackage{amssymb,graphicx,color}
\usepackage[all]{xy}

\newtheorem{theorem}{Theorem}[section]
\newtheorem{lemma}[theorem]{Lemma}

\newtheorem{proposition}[theorem]{Proposition}
\newtheorem{definition}[theorem]{Definition}
\newtheorem{example}[theorem]{Example}
\newtheorem{remark}[theorem]{Remark}

\newcommand{\R}{\mathbb R}

\newcommand{\F}{\mathbb F}

\begin{document}

\title[Lipschitz contact equivalence]{Lipschitz contact equivalence of function germs in $\R^2$}

\author{Lev Birbrair}
\thanks{L.~Birbrair was partially supported by CNPq-Brazil grant 300575/2010-6}
\address{L.~Birbrair, Departamento de Matemática, Universidade Federal do Ceará (UFC), Campus do Pici, Bloco 914, Cep. 60455-760, Fortaleza-Ce, Brazil}
\email{birb@ufc.br}

\author{Alexandre Fernandes}
\thanks{A.~Fernandes was partially supported by CNPq-Brazil grant 302998/2011-0}
\address{A.~Fernandes, Departamento de Matemática, Universidade Federal do Ceará (UFC), Campus do Pici, Bloco 914, Cep. 60455-760, Fortaleza-Ce, Brazil}
\email{alexandre.fernandes@ufc.br}

\author{Andrei Gabrielov}
\thanks{A.~Gabrielov was partially supported by NSF grant DMS-1161629}
\address{A.~Gabrielov, Department of Mathematics,
Purdue University, West Lafayette, IN 47907, USA}
\email{agabriel@math.purdue.edu}

\author{Vincent Grandjean}
\thanks{V.~Grandjean was partially supported by CNPq-Brazil grant 150555/2011-3}
\address{V.~Grandjean, Departamento de Matemática, Universidade Federal do Ceará (UFC), Campus do Pici, Bloco 914, Cep. 60455-760, Fortaleza-Ce, Brazil}
\email{vgrandje@fields.utoronto.ca}



\begin{abstract}
In this paper we study Lipschitz contact equivalence of continuous function germs
in the plane definable in a polynomially bounded o-minimal structure,
such as semialgebraic and subanalytic functions.
We partition the germ of the plane at the origin into zones where the function has explicit asymptotic behavior.
Such a partition is called a pizza. We show that each function germ admits a minimal pizza,
unique up to combinatorial equivalence.
We show then that two definable continuous function germs are definably Lipschitz contact
equivalent if and only if their corresponding minimal pizzas are equivalent.
\end{abstract}

\maketitle

\section{Introduction}

Lipschitz geometry of maps is a rapidly growing subject in  contemporary Singularity Theory.
Recent progress in this area is due to the tameness theorems proved by several researchers
(see, for example, \cite{Fukuda}, \cite{BCFR}, \cite{RV}, \cite{HP}).
However description of a set of invariants is barely developed.
This paper presents a classification of the germs of continuous function germs
at the origin of $\R^2$ definable in a polynomially bounded
o-minimal structure (e.g., semialgebraic or subanalytic functions)
with respect to the definable Lipschitz contact equivalence.
This classification is tame, unlike the Lipschitz R-equivalence (see \cite{BCFR} and \cite{RV}).
The most important ingredient of the invariant constructed here is the so-called width function.
Let $f$ be (the germ at the origin of $\R^2$ of) a continuous definable function with $f(0)=0$.
The width $\mu^*(\gamma)$  of (the germ at the origin of) a definable arc $\gamma$ with respect
to $f$ is the minimal order of contact of the ``nearby'' definable arcs  along which $f$ has
the same order as  along $\gamma$.
For any exponent $q$ of the field of exponents $\F$ of the given o-minimal structure,
we define $\mu(q)$ to be the (possibly empty) set of the widths
$\mu^*(\gamma)$ of all arcs $\gamma$ along which $f$ has order $q$.
We show that the multifunction $q\mapsto\mu(q)$ is finite.
The neighborhood of the origin can be divided into finitely many zones so that in each zone
$\mu(q)$ is a well defined (single valued) function. Moreover, this partition can be done so
that in each zone $\mu(q)$ is an affine function  with coefficients in $\F$.
This partition into zones, with the data specifying the sign of $f$ and the affine function $\mu(q)$
for each zone, is called a pizza. A pizza is not unique, but a simplification procedure described
in Section 4 provides a ``minimal'' pizza for the given function $f$,
which is unique up to natural combinatorial equivalence.
The minimal pizza provides a complete invariant for the definable contact Lipschitz equivalence class of $f$.
Our construction is based on the Preparation Theorem
for definable functions in polynomially bounded o-minimal structures (van den Dries and Speissegger \cite{DS}).
Our width function is related to the Newton Boundary of a function on an analytic arc constructed
by Koike and Parusinski \cite{KP}.

\section{Basic definitions}

\begin{definition}\label{defc0k}
{\rm We say that two continuous map germs  $f,g: (\R^n,0) \longrightarrow (\R^p,0)$ are
\emph{Lipschitz contact equivalent} if there exist two germs of
bi-Lipschitz homeomorphisms $h:(\R^n,0) \longrightarrow (\R^n,0)$ and  $H: (\R^n
\times \R^p,0) \longrightarrow (\R^n \times \R^p,0)$ such that
$H(\R^n \times \{0 \}) = \R^n \times \{0 \}$ and the following
diagram is commutative:}

\begin{equation}\label{diagram}
\begin{array}{lllll}
(\R^n,0) & \stackrel{(id,\, f)}{\longrightarrow}
 & (\R^n \times
\R^p,0) & \stackrel{\pi_n}{\longrightarrow} & (\R^n,0) \\
\,\,\,h \, \downarrow &  & \,\,\,\, H \, \downarrow  &  & \,\,\, h \, \downarrow \\
(\R^n,0)& \stackrel{(id, \,g)}{\longrightarrow} & (\R^n \times
\R^p,0) & \stackrel{\pi_n}{\longrightarrow}& (\R^n,0) \\
\end{array}
\end{equation}

\medskip

\noindent {\rm where  $id:\R^n  \longrightarrow \R^n$ is the identity
mapping and $\pi_n: \R^n \times \R^p \longrightarrow \R^n$ is the
canonical projection.}
\end{definition}

In this paper we consider the case $p=1$, thus the maps $f,\,g$ are functions.
There is a more convenient way to work with the contact equivalence of functions, due to the following

\begin{theorem}[\cite{BCFR}]\label{thm:contact-equiv}
 Let $f$ and $g$ be two Lipschitz contact equivalent continuous function germs $(\R^n,0) \to (\R,0)$.
 Then there exists a germ at the origin of a bi-Lipschitz
homeomorphism $\Phi:(\R^n,0) \longrightarrow (\R^n,0)$ such that

\smallskip
\noindent$(\star)$ either $a f \leq  g\circ\Phi \leq b f$, or $a f \leq - g\circ\Phi \leq b f$,
for some positive constants $a$ and $b$.

\smallskip
\noindent If $f$ and $g$ are Lipschitz and satisfy $(\star)$
then they are Lipschitz contact equivalent.
\end{theorem}

\medskip
For the rest of the paper, we assume $n=2$.

\medskip
In this paper, we consider a polynomially bounded o-minimal structure $\mathcal A$ over $\R$,
with the field of exponents $\F$.
We denote $\F_+$ the set of positive exponents in $\F$.
All functions are assumed to be definable in $\mathcal A$,
and the Lipschitz contact equivalence is assumed to be definable.
This means that $h$ and $H$ in (\ref{diagram}) are definable in $\mathcal A$.
A function $f(x,y)$ is always identified with its germ at the origin of $\R^2$.

An arc $\gamma$ is a continuous definable mapping $\gamma:[0,\epsilon)\to\R^2$ such that $\gamma (0) =0$.
Unless explicitly stated otherwise,
an arc is parameterized by the distance to the origin, i.e., $|\gamma(t)|=t$.
We always consider $\gamma$ as a germ at the origin of $\R^2$.
When it does not lead to confusion, we use the same notation for an arc and its image in $\R^2$.

\begin{definition}{\rm The \emph{order of tangency} $\rm{tord}(\gamma_1,\gamma_2)$ of two distinct arcs $\gamma_1$
and $\gamma_2$, is the exponent $\beta\in \F,\;\beta\ge 1$, defined in the following equation}
 $$|\gamma_1(t)-\gamma_2(t)|=b t^\beta+o(t^\beta),\quad b\ne 0.$$
\end{definition}

\begin{definition}{\rm Let $f\colon(\R^2,0)\rightarrow(\R,0)$ be a continuous function, and $\gamma$ an arc in $\R^2$.
 If $f|_\gamma\not\equiv 0$, \emph{the order of $f$ along $\gamma$}, denoted by $\rm{ord}_{\gamma}(f)$,
 is defined as the exponent $\alpha\in\F_+$ in
$$f(\gamma(t))= at^{\alpha}+o(t^{\alpha}),\quad a\ne 0.$$
If $f|_\gamma\equiv 0$, we set $\rm{ord}_{\gamma}(f)=\infty$.}
\end{definition}

\begin{definition}{\rm
Two arcs $\gamma_1$ and $\gamma_2$ divide the germ of $\R^2$ at the origin into two components.
If $\beta=\rm{tord}(\gamma_1,\gamma_2)>1$ then the closure of the smaller (not containing a half-plane) component
is called a \emph{$\beta$-H\"older triangle}.
If $\rm{tord}(\gamma_1,\gamma_2)=1$ then the closure of each of the two components is called
a \emph{$1$-H\"older triangle}. The number $\beta\in\F$ is called the \emph{exponent} of the H\"older triangle.
The arcs $\gamma_1$ and $\gamma_2$ are called the \emph{sides} of the H\"older triangle.
We denote by $T(\gamma_1,\gamma_2)$ a H\"older triangle bounded by $\gamma_1$ and $\gamma_2$.}
\end{definition}

 Let $T\subset (\R^2,0)$ be a  H\"older triangle, and let $f\colon T\rightarrow(\R,0)$ be a continuous function.
Define
\begin{equation}\label{Q}
Q_f(T)=\bigcup_{\gamma\subset T}\rm{ord}_{\gamma}(f).
\end{equation}

\begin{proposition} For a H\"older triangle $T$, $Q_f(T)$ is a segment  in $\F_+\cup\{\infty\}$.
\end{proposition}

\begin{proof} Suppose that $q_1,q_2\in Q_f(T)$ and let $q\in (q_1,q_2)\cap\F_+$.
Let $h\colon (\R^2,0)\rightarrow(\R,0)$ be a continuous function defined by $h(x,y)=(x^2+y^2)^{q/2}$.
Since the intersection of the graphs of $f_{| T}$ and $h_{| T}$, as a germ at $0\in\R^3$,
does not reduce to the origin, the arc-selection lemma implies that there exists an arc $\gamma$ in $T$
such that $\rm{ord}_{\gamma}(f)=q$.
\end{proof}

We will show later that $Q_f(T)$ is a closed segment.

\begin{definition} {\rm A H\"older triangle $T$ is called \emph{elementary} with respect to the function $f$ if, for any two
disjoint arcs $\gamma_1$ and $\gamma_2$ in $T$ such that $\rm{ord}_{\gamma_1}(f)=\rm{ord}_{\gamma_2}(f)=q$,
the order of $f$ is $q$ on any arc in the H\"older triangle $T (\gamma_1,\gamma_2)\subset T$.}
\end{definition}

\begin{definition}\label{df:width}
{\rm Let $f\colon(\R^2,0)\rightarrow(\R,0)$ be a continuous function. For each arc $\gamma$, \emph{the width} of $\gamma$
with respect to $f$ is the infimum $\mu^*(\gamma,f)$ of the exponents of H\"older triangles $\tilde T$ containing $\gamma$
such that $Q_f(\tilde T)$ is a point.}

{\rm Let $T$ be a H\"older triangle. The \emph{relative width} of an arc $\gamma\subset T$, with respect to $f$
and $T$, is the infimum $\mu_T^*(\gamma,f)$ of the exponents of H\"older triangles $\tilde{T}$
such that $\gamma\subset\tilde{T}\subset T$ and $Q_f(\tilde T)$ is a point.}

{\rm The multivalued \emph{width function} $\mu_{T,f}\colon Q_f(T)\rightarrow\F\cup\{\infty\},\;\mu_{T,f}\ge\beta,$ is defined as follows.
For $q\in Q_f(T)$, we define $\mu_{T,f}(q)$ as the (finite) set
 of exponents $\mu_T^*(\gamma,f)$, where $\gamma$ is any arc in $T$ such that $\rm{ord}_\gamma(f)=q$.}
\end{definition}

We will show (see Lemma \ref{lem:width-affine} below) that these infima $\mu^*(\gamma,f)$ and $\mu^*_T(\gamma,f)$
are both minima and belong to $\F_+\cup\{\infty\}$.

\begin{remark}\label{rem:order-width}
{\rm Let $f,g : (\R^2,0)\to (\R,0)$ be two continuous function germs which are Lipschitz contact equivalent.
Let $\Phi$ be the bi-Lipschitz homeomorphism of Theorem \ref{thm:contact-equiv}.
For any arc $\gamma$, let $\tilde\gamma=\Phi(\gamma)$.
Then, $\rm{ord}_{\tilde\gamma} (g) = \rm{ord}_\gamma (f)$ and $\mu^*(\tilde\gamma,g) = \mu^*(\gamma,f)$.}
\end{remark}

\medskip\noindent
{\bf Notation.}  When the function germ $f$ is fixed, we write $\mu^*(\gamma)$
and $\mu_T^*(\gamma)$ instead of $\mu^*(\gamma,f)$ and $\mu_T^*(\gamma,f)$,
respectively. We also write $\mu_T$ instead of $\mu_{T,f}$.

\begin{remark}
{\rm If $T$ is an elementary triangle then $\mu_T$ is single valued.}
\end{remark}

\begin{definition}{\rm  A \emph{H\"older complex} on $\R^2$ is a (definable) triangulation of the germ of $\R^2$
at the origin. Two H\"older complexes are \emph{combinatorially equivalent} when there exists
a bijection between their sets of triangles that either preserves or reverses their cyclic order,
and preserves their H\"older exponents (see \cite{B}). A combinatorial type of a H\"older complex can be defined as a finite sequence of exponents $\beta_i\in\mathbb{F}$; $\beta_i\geq 1$, considered with the cyclic order.
At least one of the exponents $\beta_i$ is equal to $1$.
The sequence $\{\beta_i\}$ is called \emph{an abstract H\"older complex}.
A H\"older complex $\{T_i\}$ corresponds to an abstract H\"older complex $\{\beta_i\}$ if the exponent of $T_i$ is equal to $\beta_i$, for all $i$. }
\end{definition}

\begin{definition}\label{abstract}
{\rm An \emph{abstract pizza} is a finite collection $\mathcal{H}=\{\beta_i,Q_i,s_i,\mu_i\}_{i\in I}$,
where $I=\{1,\dots,k\}\mod{k}$ is considered with the cyclic order, and}
\begin{enumerate}
\item {\rm $\{\beta_i\}$ is an abstract H\"older complex  on $\R^2$ at the origin};
\item {\rm each $Q_i$ is a closed directed segment of $\mathbb{F}_+\cup\{\infty\}$, where ``directed'' means that $Q_i=[a_i,b_i]$ with either $a_i<b_i$ or $a_i>b_i$ (or $a_i=b_i$ when $Q_i$ is a point) satisfying the continuity condition $a_{i+1}=b_i$ for all $i$;}
\item {\rm each $s_i$ is a sign $+$, $-$ or $0$, with $s_i=s_{i+1}$ unless $b_i=a_{i+1}=\infty$;}
\item {\rm $\mu_i\colon Q_i\rightarrow\mathbb{F}\cup\{\infty\}$ is an affine function, such that $\min(\mu_i(a_i),\mu_i(b_i))=\beta_i$ for each $i$.}
\end{enumerate}
\end{definition}

\begin{definition}{\rm A pizza $\mathcal{H}=\{\beta_i,Q_i,s_i,\mu_i\}_{i\in I}$
is associated with a continuous function germ $f\colon(\R^2,0)\rightarrow(\R,0)$ if
there exists a H\"older complex $\{T_i\}_{i\in I}$ on $\R^2$ where each $T_i=T(\gamma_i,\gamma_{i+1})$ is a
$\beta_i$-H\"older triangle elementary with respect to $f$,
and the arcs $\gamma_i$ are either
counterclockwise or clockwise oriented with respect to the cyclic order on $I$, such that}
\begin{enumerate}
\item {\rm $Q_i=Q(T_i)$;}
\item {\rm for each arc $\gamma\subset T_i$, $\mu_{T_i}^*(\gamma)=\mu_i(\rm{ord}_{\gamma}(f))$;}
\item {\rm the sign of $f$ on the interior of $T_i$ is $s_i$.}
\end{enumerate}
\end{definition}

\begin{definition}{\rm Two pizzas $\mathcal{H}=\{\beta_i,Q_i,s_i,\mu_i\}_{i=1}^{k}$ and $\mathcal{H}'=\{\beta'_j,Q'_j,s'_j,\mu'_j\}_{j=1}^{k'}$ (see Fig.~\ref{equiv}) are called \emph{combinatorially equivalent} (or simply {\em equivalent}) if $k=k'$ and there is a combinatorial equivalence $i\mapsto j(i)$ of the corresponding H\"older complexes associating $T'_{j(i)}$ to $T_i$, such that}
\begin{enumerate}
\item {\rm either $s'_{j(i)}=s_i$ or $s'_{j(i)}=-s_i$ for all $i$;}
\item {\rm $Q'_{j(i)} = Q_i$ for all $i$ if $i\mapsto j(i)$ preserves the cyclic order,
or $Q'_{j(i)} = -Q_i$ for all $i$ (where $-Q_i$ means $Q_i$ with the opposite direction) if the cyclic order
is reversed;}
\item $\mu'_{j(i)}=\mu_i$ for all $i$.
\end{enumerate}
\end{definition}
 \begin{figure}
 \centering
 \includegraphics[width=5in]{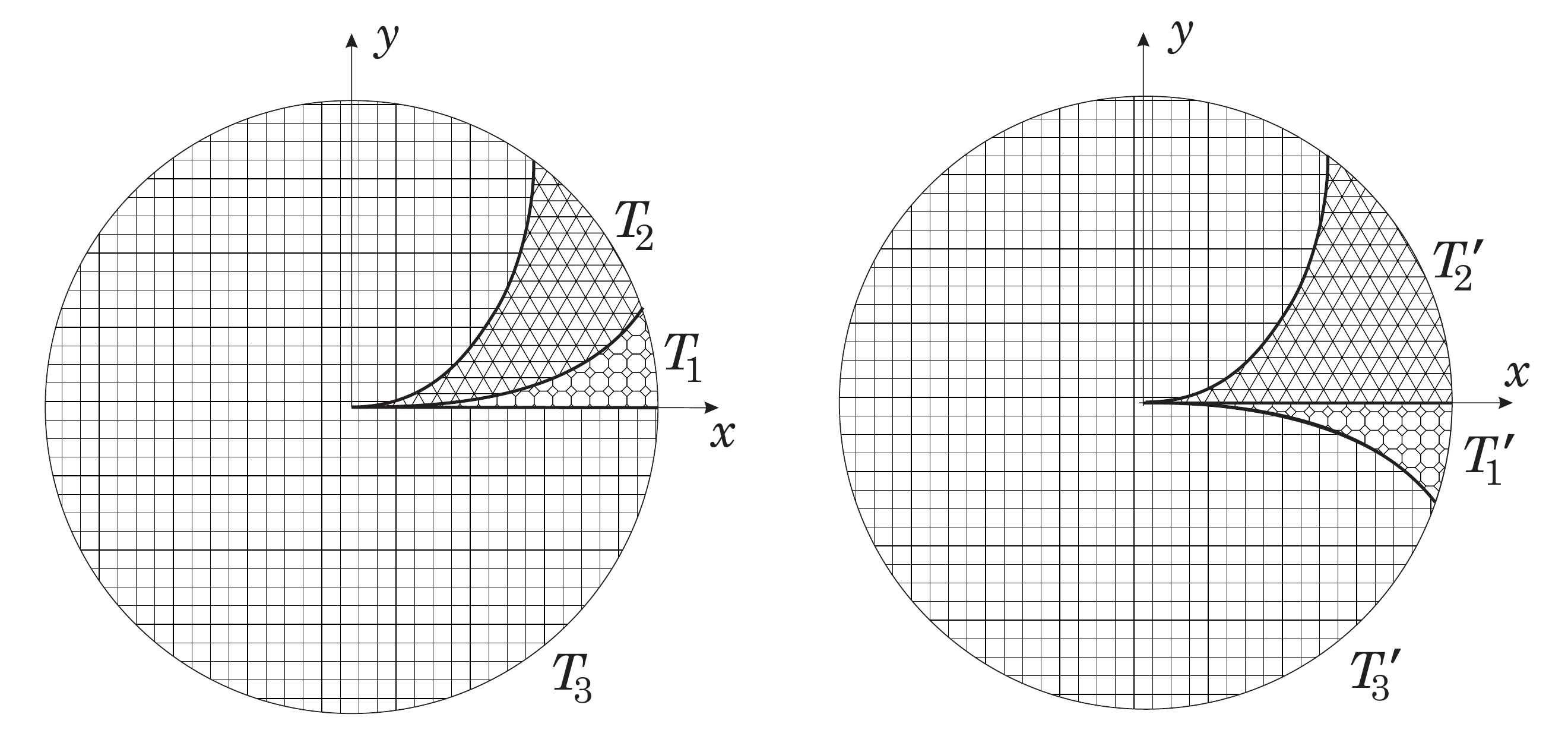}
 \caption{Equivalence of pizzas.}\label{equiv}
 \end{figure}

\section{Main theorem.}
\begin{theorem}\label{theorem1} For any continuous definable function germ $f\colon(\R^2,0)\rightarrow(\R,0)$, there exists a pizza associated with $f$.
\end{theorem}

\begin{proof}
The existence of a pizza associated with $f$ uses a special case of the Preparation Theorem of van den Dries and Speissegger \cite{DS}.
Namely,

\begin{theorem}[\cite{DS}]\label{thm:prep}
Let $f\colon(\R^2,0)\rightarrow(\R,0)$ be a definable and continuous function.
There exists a finite decomposition $\mathcal{C}$ of $\R^2$, as a germ at $0$, and for each $T\in\mathcal{C}$
there exists an exponent $\lambda\in\mathbb{F}$ and definable functions $\theta,a\colon(\R,0)\rightarrow\R$
and $u\colon(\R^2,0)\rightarrow\R$, such that for $(x,y)\in T$ we have
\begin{equation}\label{prep}
f(x,y)=(y-\theta(x))^{\lambda}a(x)u(x,y), \quad |u(x,y)-1|<\frac{1}{2}.
\end{equation}
Up to refining, we can further require that the set $\{y=\theta(x)\}$ is either outside $T$
or on its boundary.
\end{theorem}

The Preparation Theorem \ref{thm:prep} specifies a special direction, that of the variable $y$, with respect to which we
can prepare the function of interest in the form given in Equation (\ref{prep}), mimicking the classical Weierstrass
Preparation for a complex function germ.
Thus we get the decomposition $\mathcal{C}_y$ into definable cells.
Preparing the function with respect to the direction of the variable $x$ is also possible, but gives
rise to a second decomposition $\mathcal{C}_x$, different from $\mathcal{C}_y$.
Nevertheless we can refine $\mathcal{C}_y$
so that each cell of the refined decomposition $\mathcal{C}$ is contained in a cell of $\mathcal{C}_x$.
Thus the function $f$ may be prepared with respect to both $x$-direction and $y$-direction
in each cell of $\mathcal{C}$.

We may further assume that each cell $C$ of $\mathcal{C}$ satisfies the
following property: either there is no arc contained in $C$ tangent to the $y$-axis, or there is no arc contained in $C$ tangent to the $x$-axis.

\smallskip
Let $C$ be a cell of $\mathcal{C}$. Up to permuting the $x$ and $y$ coordinates,
we can assume that the function $f$ is prepared in $C$ with respect to the $y$-direction, there is no arc
in $C$ tangent to the $y$-axis, and the curve $\beta=\{y=\theta(x)\}$ is not tangent to the $y$-axis.
A simple but important consequence of this property of $C$ is that there is a positive constant $K$
such that for $(x,y)\in C$, we have
\begin{center}
$|(x,y)| \ll 1 \Longrightarrow |(x,y)| \leq K |x|$.
\end{center}
Then, for any arc $t\to\gamma(t) = (x(t),y(t)) \in C$, we have $|x(t)| \leq t \leq K |x(t)|$.

\smallskip
Since there is no arc in $C$ tangent to the $y$-axis, we can assume that $C$ is contained in
the half-plane $\{x\geq 0\}$.

Let $T$ be the closure of $C$, and $\gamma$ an arc in $T$.
Then $\gamma=\{y=\theta(x)+b\,x^{\rm{tord}(\gamma,\beta)}+o(x^{\rm{tord}(\gamma,\beta)})\}$.
Since $a(x)=c\,x^r+o(x^r)$, Equation (\ref{prep}) implies
\begin{equation}\label{order-f}
\rm{ord}_\gamma(f)=\lambda \cdot \rm{tord}(\gamma,\beta) + r.
\end{equation}
Let $R(T):=[\rm{tord}(\gamma_1,\beta),\,\rm{tord}(\gamma_2,\beta)]\cap (\mathbb{F}\cup \{+\infty\})$,
where $\gamma_1$ and $\gamma_2$ are the boundary arcs of $T$.
If $R(T)$ consists of a single point, or if $\lambda=0$,
then $Q_f(T)$ is a single point. Otherwise,
we define the function $\rho: R(T) \to \mathbb{F}\cup \{+\infty\}$
as $\rho(q): =(q-r)/\lambda$.  It is an affine function on $Q_f(T)$.  Note that
$\rho(q)=\rm{tord}(\gamma,\beta)$ for any arc $\gamma\subset T$ such that $\rm{ord}_\gamma(f)=q$.

\begin{lemma}\label{lem:width-affine}
The following equality holds: $\rho(q)=\mu_T(q)$ for all $q\in Q_f(T)$.
\end{lemma}

\begin{proof} Suppose that $\mu_T(q)<\rho(q)$ (see Fig.~\ref{fig2}a). Let $\gamma$ be an arc in $T$ such that
$\rm{ord}_\gamma(f)=q$ and $\mu_T^*(\gamma)<\rm{tord}(\gamma,\gamma_1)$,
where $\gamma_1$ is the side of $T$ closest to $\beta$.

Then, there exists an arc $\tilde{\gamma}$, such that $\rm{tord}(\gamma,\tilde{\gamma})=\mu_T^*(\gamma)$, $\rm{ord}_{\tilde{\gamma}}(f)=q$ and $\rm{tord}(\tilde{\gamma},\gamma_1)\neq\rm{tord}(\gamma,\gamma_1)$. It contradicts the fact that $\rho(q)$ is single valued.

Suppose that $\mu_T(q)>\rho(q)$ (see Fig.~\ref{fig2}b).
Let $\gamma$ be an arc in $T$ such that $\rm{ord}_\gamma(f)=q$ and $\mu_T^*(\gamma)>\rm{tord}(\gamma,\gamma_1)$.

Then, one can consider an arc $\tilde{\gamma}$ in $T$ such that
$\rm{tord}(\tilde{\gamma},\gamma)=\rm{tord}(\tilde{\gamma},\gamma_1)=\rm{tord}(\gamma,\gamma_1)$.
Since $T$ is an elementary triangle, one cannot have $\rm{ord}_{\tilde{\gamma}}(f)=\rm{ord}_\gamma(f)$.
But this also contradicts the fact that $\rho(q)$ is single valued.
\end{proof}

\begin{figure}
\centering
\includegraphics[width=5in]{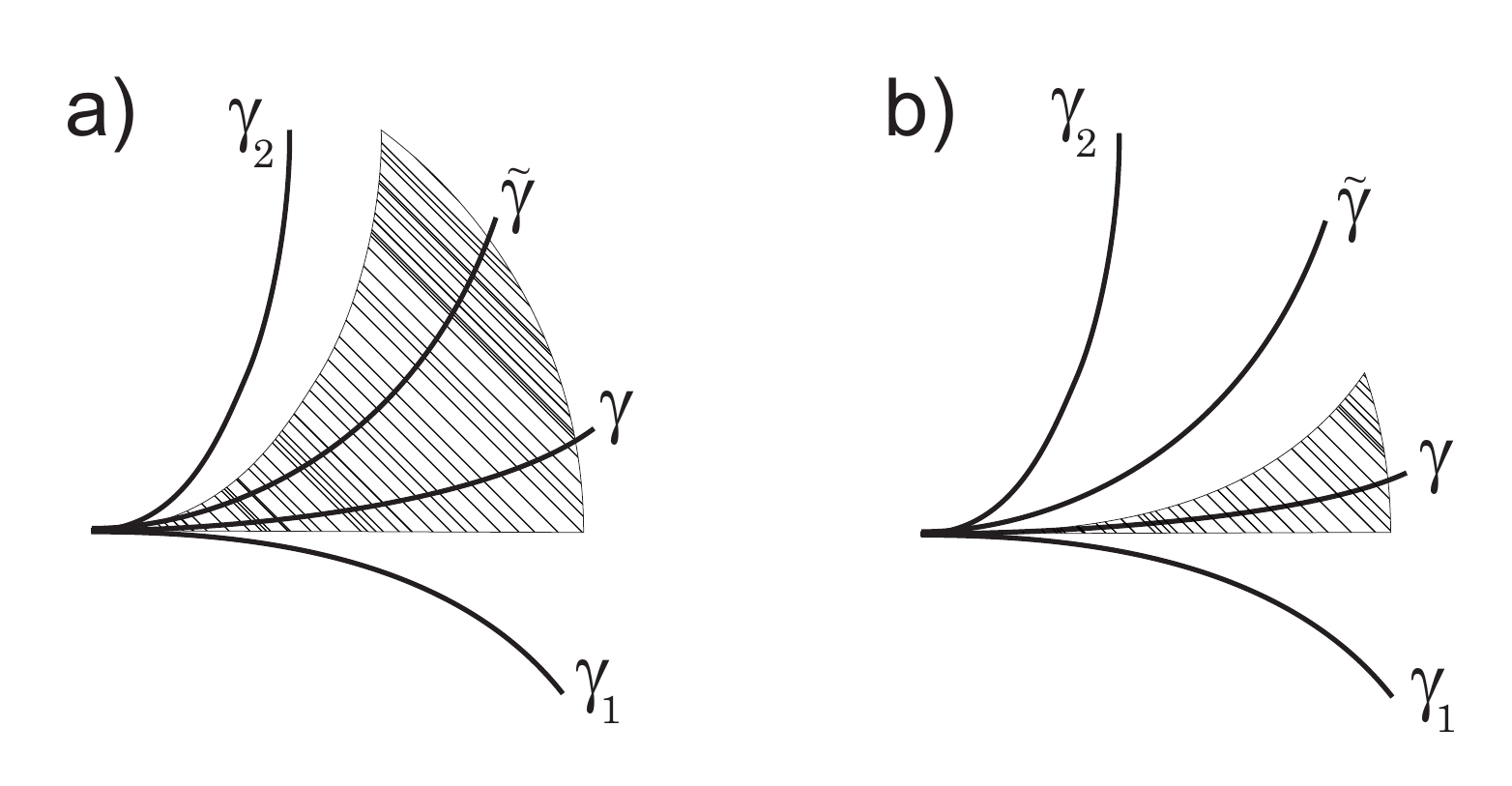}
\caption{Arcs $\gamma$ and $\tilde\gamma$ in the proof of Lemma \ref{lem:width-affine}. a) Case $\mu_T(q)<\rho(q)$. b) Case $\mu_T(q)>\rho(q)$.}\label{fig2}
\end{figure}

In Lemma \ref{lem:width-affine} we constructed a H\"older complex such that each triangle $T$
is elementary with respect to the function $f$,
the width function $\mu_T:\mathbb{F}\cup\{\infty\}\to\mathbb{F}\cup\{\infty\}$, referred to as $\rho$ in the proof of Lemma \ref{lem:width-affine}, is affine, $\mu_T\ge\beta$
where $\beta$ is the exponent of $T$. The sign of $f$ inside $T$ is clearly fixed.
If $Q_f(T)$ is not a point then $\mu_T$ is not constant. If $T=T_i$ is bounded by the arcs $\gamma_1$ and $\gamma_2$ so that the pair $\gamma_1$, $\gamma_2$ is counterclockwise oriented,
we set $a_i=\rm{ord}_{\gamma_1}(f)$ and $b_i=\rm{ord}_{\gamma_2}(f)$, $Q_i=[a_i,b_i]$.
The continuity condition $a_{i+1}=b_i$ follows from the continuity of $f$.
This completes the proof of Theorem \ref{theorem1}.
\end{proof}

\begin{remark}
{\rm The proof of Lemma \ref{lem:width-affine} shows that the width $\mu_T^*(\gamma,f)$ is a minimum,
i.e., there exists an arc $\gamma'$ in $T$ such that $\mu_T^*(\gamma,f)={\rm tord}(\gamma,\gamma')$.}
\end{remark}

\begin{theorem}\label{theorem2} Let $f,g\colon(\R^2,0)\rightarrow(\R,0)$ be germs of continuous definable functions.
If $f$ is Lipschitz contact equivalent to $g$, then for each pizza
$\mathcal{H}=\{\beta_i,Q_i,s_i,\mu_i\}_{i=1}^{k}$ associated with $f$ there is a pizza
$\mathcal{H}'$ associated with $g$, and equivalent to $\Pi$.
\end{theorem}

\begin{proof}[Proof of Theorem \ref{theorem2}] Let $\{T_i\}$ be a triangulation of the germ of $\R^2$ at zero corresponding to Definition 2.13.  Let $(H,h)$ be  a pair of bi-Lipschitz homeomorphisms
defining the Lipschitz contact equivalence between $f$ and $g$. We have the relation
$H((x,y),f(x,y)) = (h(x,y),g(h(x,y)))$.
Since $h$ is a bi-Lipschitz map, $T'_i=h(T_i)$ is also a $\beta_i$-H\"older triangle.
Let $\gamma$ be a definable arc in $T_i$. Since $H$ is also bi-Lipschitz,
$\rm{ord}_{\gamma}(f)=\rm{ord}_{h(\gamma)}(f)$.
Let $\gamma_1$ and $\gamma_2$ be two arcs in $T_i$. Since $H$ is
a bi-Lipschitz homeomorphism, $\rm{ord}_{\gamma_1}(f)=\rm{ord}_{h(\gamma_1)}(g)$ and
$\rm{ord}_{\gamma_2}(f)=\rm{ord}_{h(\gamma_2)}(g)$. Thus
$Q_f(T_i)=Q_g(T'_i)$ (as directed segments).
If $T_i$ is an elementary triangle with respect to $f$, then $T'_i$ is an
elementary triangle with respect to $g$, and $\mu_i$ is the width function for $T'_i$.
Note that, if the map $H$ preserves (respectively, reverses) the sign
of $f$ on some triangle, then it has to preserve (respectively, to reverse) the sign on each
triangle. Thus there exists a pizza $\mathcal{H}'$ associated with $g$ having all elements
same as $\mathcal{H}$ except, possibly, all signs $s_i$ reversed.
This completes the proof of Theorem \ref{theorem2}.
\end{proof}

\section{Simplification of Pizzas}

Let $\mathcal{H}=\{\beta_i,Q_i,s_i,\mu_i\}$ be an abstract  pizza. A \emph{ simplification } of $\mathcal{H}$ is a pizza $\tilde{\mathcal{H}}$ obtained from $\mathcal{H}$ using the following operations:

\smallskip
{\bf 1}. Let $\beta_i$ and $\beta_{i+1}$ be two consecutive  numbers of the formal  H\"older complex  of $\mathcal{H}$.  Suppose that $Q_i$ and $Q_{i+1}$ are not single points, and the following holds:
\begin{enumerate}
\item $s_i=s_{i+1}$;
\item $Q_i=[a_i,b_i]$, $Q_{i+1}=[a_{i+1},b_{i+1}]$  and either
 $a_i<b_i = a_{i+1}<b_{i+1}$ or $a_i>b_i = a_{i+1}>b_{i+1}$;
\item There exists an affine function $\tilde{\mu}: \mathbb{F} \to \mathbb{F}$ such that
 $\mu_j:= \tilde{\mu}|_{Q_j}$ for $j=i,i+1$.
\end{enumerate}

Then, we define a new pizza as follows:
\\
- For $j \leq i-1$ we set $\widetilde{\beta}_j := \beta_j$, $\tilde{\mu}_j := \mu_j$, $\tilde{s}_j := s_j$ ;
\\
- For $i+2 \leq j \leq k$  we set $\widetilde{\beta}_j := \beta_{j+1}$, $\tilde{\mu}_j := \mu_{j+1}$, $\tilde{s}_j := s_{j+1}$ ;
\\
- We define $\widetilde{\beta}_i := \min\{\beta_i,\beta_{i+1}\}$, $\tilde{s}_i := s_i=s_{i+1}$ and $\tilde{\mu}_i := \tilde{\mu}|_{Q_i\cup Q_{i+1}}$.

The new abstract pizza  now has only $k-1$ triangles instead of $k$.

\smallskip
\begin{remark} {\rm Notice that if $a_i<b_i$ and $a_{i+1}>b_{i+1}$ or $a_i>b_i$ and $a_{i+1}<b_{i+1}$, then we do not apply the simplification procedure.}
\end{remark}

\smallskip
{\bf 2}. Let $\beta_i$ and $\beta_{i+1}$ be a pair of consecutive numbers  in the formal H\"older complex of $\mathcal{H}$  such that at least one of the segments $Q_i$ and $Q_{i+1}$ is a point. Suppose that  $Q_i=[a,a]$,  $Q_{i+1}=[a,b]$ and $\beta_i\geq \mu_{i+1}(a)$. Then, we define $\tilde{\beta}_j$, $\tilde{s}_j$ and $\tilde{\mu}_j$ for $j\not\in\{i,i+1\}$, in the same way as in the previous case, and set $\tilde{\beta}_i=\beta_{i+1}$, $\tilde{s}_i := s_i=s_{i+1}$, $\tilde{Q}_i=Q_{i+1}$ and $\tilde{\mu}_i=\mu_{i+1}$.

If $Q_i=[a,b]$ and $Q_{i+1}=[b,b]$, the procedure is almost the same as before, the only difference is that we set $\tilde{\beta}_i=\beta_i$, $\tilde{Q}_i=Q_i$ and $\tilde{\mu}_i=\mu_i$.

A pizza is called \emph{simplified} if none of the operations above can be applied.
Any pizza can be simplified  applying the operations {\bf 1} and {\bf 2}.

\begin{proposition}
The combinatorial equivalence class of a resulting simplified pizza does not depend on the order of simplifications.
\end{proposition}

\begin{proof} If we apply the simplification procedure until it cannot be applied, any two consecutive elements
 indexed by $i$ and $i+1$ must have one of the following properties.

\begin{enumerate}
\item The affine functions $\mu_i$ and $\mu_{i+1}$ are non-constant and they are not restrictions of the same affine function to two adjacent segments;
\item $Q_i=[a,a]$ is a point, $Q_{i+1}=[a,b]$ is not a point, and $\mu_i(a)<\mu_{i+1}(a)$;
\item $Q_i=[a,b]$ is not a point, $Q_{i+1}=[b,b]$ is a point, and $\mu_i(b)>\mu_{i+1}(b)$.
\end{enumerate}

The corresponding maximal segments are unique.
Their order depends only on the initial pizza, and does not depend on the simplification procedure.
\end{proof}

The pizza $\tilde{\mathcal{H}}$ obtained from $\mathcal{H}$ by the operations described above is called a \emph{simplification} of $\mathcal{H}$. The pizza $\mathcal{H}$ is called a \emph{refinement} of $\tilde{\mathcal{H}}$.

In geometric terms, the simplification procedure can be described as follows. Let us consider the germ of a definable continuous function $f$ and an abstract pizza associated with $f$. Let $\{T_i\}$ be the corresponding triangulation of the germ $(\R^2,0)$. Suppose that the width functions of two consecutive H\"older triangles $T_i$ and $T_{i+1}$ is same affine function, let us say $\tilde{\mu}: \mathbb{F} \to \mathbb{F}$,  restricted to adjacent segments $Q(T_i)$ and $Q(T_{i+1})$. Then, one considers a union of these H\"older triangles as a H\"older triangle with the minimal exponent. The width function of the new triangle is the restriction of $\mu$ to $Q(T_i)\cup Q(T_{i+1})$. This proves the following result.

\begin{lemma}\label{lemma4}
Let $\mathcal{H}$ be a pizza associated with a function $f$. If $\widetilde{\mathcal{H}}$ is a simplification of $\mathcal{H}$,
then $\widetilde{\mathcal{H}}$ is also a pizza associated with $f$.
\end{lemma}

The last Lemma allows us to define a notion of a \emph{minimal pizza} associated with a function $f$ as a
simplification of any pizza associated with $f$.

\begin{example}\label{example1}
{\rm Let us define $f$ as $f = x^4+y^2$ if $x\ge 0$ and $f = x^2 + y^2$ if $x\le 0$.

For $\alpha\ge 1$, let $\gamma=\{y=a x^\alpha+o(x^\alpha),\;x\ge 0\}$ be an arc parameterized by $x$.
If $\alpha\ge 2$ then $\rm{ord}_\gamma(f)=4$, otherwise $\rm{ord}_\gamma(f)=2\alpha$.
This implies that $Q_f(T)=[4,4]$ and $\mu_{T}(4)=2$ for any H\"older triangle $T$ bounded by
$\gamma_1=\{y=a_1 x^2+o(x^2),\;x\ge 0\}$ and
$\gamma_2=\{y=a_2 x^2+o(x^2),\;x\ge 0\}$, and $Q_f(T')=[2,2]$ and $\mu_{T'}(2)=1$ for
any H\"older triangle $T_3$ bounded by two arcs not tangent to the positive $x$-axis
and containing the negative $x$-axis.

Any H\"older triangle $T_1$ bounded by an arc $\gamma_1=\{y=a x^\alpha+o(x^\alpha),\;x\ge 0,\,\alpha\ge 2\}$
and an arc $\gamma_2$ not tangent to the positive $x$-axis is elementary, with $Q_f(T_1)$ either $[2,4]$ or $[4,2]$,
and $\mu_{T_1}(q)=q/2$. The minimal pizza for $f$ consists of any such triangle $T_1$
and its complementary triangle $T_2$ bounded by the same two arcs, with $Q_f(T_2)=-Q_f(T_1)$
(the two segments have opposite directions) and $\mu_{T_2}(q)=q/2$.
Any two such pizzas are equivalent.}
\end{example}
\begin{example}\label{example2}
{\rm Let us define $f$ as $f= x^4+y^2$ if $y\ge 0$, and $f= x^4 + y^4$ if $y\le 0$.

Any H\"older triangle $T_1$ bounded by an arc $\gamma_1=\{y=a x^\alpha+o(x^\alpha),\;x\ge 0,\,\alpha\ge 2\}$
and an arc $\gamma_2$ in the upper half plane not tangent to the $x$-axis is elementary, with
$Q_f(T_1)=[4,2]$ and $\mu_{T_1}(q)=q/2$. The minimal pizza for $f$ consists of any such triangle $T_1$, a triangle $T_2$ in the upper half plane bounded by $\gamma_2$ and an arc
$\gamma_3=\{y=b|x|^\alpha+o(|x|^\alpha),\;x\le 0,\,\alpha\ge 2\}$, with $Q_f(T_2)=[4,2]$ and $\mu_{T_2}(q)=q/2$,
and a triangle $T_3$ bounded by the arcs $\gamma_3$ and $\gamma_1$ and containing the negative $y$-axis,
with $Q_f(T_3)=[4,4]$ and $\mu_{T_3}(4)=1$.
Note that $\mu_{T_1}(4)\ne\mu_{T_3}(4)\ne\mu_{T_2}(4)$.
Any two such pizzas are equivalent.}
\end{example}
\begin{example}\label{example3}
{\rm Let us define $f$ as $f = y^2-x^3$ for $x\ge 0$ and $x^2+y^2$ for $x\le 0$.
The function $f$ is invariant under the symmetry $(x,y) \to (x,-y)$. For simplicity,
we define a decomposition of
the upper half-plane $\{y\ge 0\}$ and complete it using the symmetry.

Let $\gamma_1=\{y=x^{3/2}, x\geq 0\}$ be the zero set of $f$ in the upper half-plane.
Let $T$ (resp., $T'$) be the H\"older triangle in the upper half-plane
bounded by $\gamma_1$ and the positive (resp., negative) $x$-axis.
Then $Q_f(T)=[3,\infty]$ and $Q_f(T')=[\infty,2]$.
Both $T$ and $T'$ are elementary triangles, with $f<0$ in $T$ and $f>0$ in $T'$.
If $\gamma=\{y=x^{3/2}+a x^\alpha+o(x^\alpha),\;x\ge 0,\,a\ne 0\}$
where $\alpha\ge 3/2$, then $q=\rm{ord}_\gamma(f)=3/2+\alpha$.
If, however, $a>0$ and $1\le\alpha\le 3/2$, then $q=\rm{ord}_\gamma(f)=2\alpha$.
This implies that $\mu_{T}(q)=q-3/2$, but $\mu_{T'}(q)$ is not affine.
If we partition $T'$ by an arc $\gamma_2=\{y=a x^{3/2}+o(x^{3/2}),\;a>1,\;x\ge 0\}$
into triangles, $T_2$ bounded by $\gamma_1$ and $\gamma_2$, and $T_3$
bounded by $\gamma_2$ and the negative $x$-axis, then
$Q_f(T_2)=[\infty,3]$, $\mu_{T_2}(q)=q-3/2$, $Q_f(T_3)=[3,2]$ and $\mu_{T_3}(q)=q/2$,
thus $\mu(q)$ is affine in both $T_2$ and $T_3$.
The minimal pizza for $f$ consists of triangles $T_1=T,\,T_2,\,T_3$ and
their symmetric triangles in the lower half-plane.
Note that the positive $x$-axis in this decomposition can be replaced by any arc
$\gamma=\{y=a x^\alpha,\;x\ge 0\}$ where either $\alpha>3/2$ or $\alpha=3/2$ and $|a|<1$,
and the negative $x$-axis can be replaced by any arc that is not tangent to the
positive $x$-axis.}
\end{example}
\begin{example}\label{example4}
{\rm Although the function $q\mapsto\mu(q)$ in Examples \ref{example1}-\ref{example3}
is always increasing in $q$, that is not always the case.
Consider, for example, $g(x,y)=(x^6+y^6)/f(x,y)$ where
$f(x,y)$ is the function from Example \ref{example1}.
Since $\rm{ord}_\gamma(x^6+y^6)=6$ for any arc $\gamma$, we have
$\rm{ord}_\gamma(g)=6-\rm{ord}_\gamma(f)$ for any $\gamma$.
This implies that a pizza for $g$ can be obtained from the pizza for $f$
by replacing $\mu_T(q)$ with $\mu_T(6-q)$ for any triangle $T$.
In particular, for any of the two triangles $T_1$ and $T_2$ in Example \ref{example1},
the function $\mu=q/2$ should be replaced with $\mu=3-q/2$.}
\end{example}

Although we just saw that a pizza associated with a function germ can never be unique,
the next section ensures that a minimal pizza is unique up to combinatorial equivalence.

The procedure of geometric refinement may be described by the same way as geometric simplification. We take a pizza $\tilde{\mathcal{H}}$ associated with the germ of a definable continuous function $f$. Suppose that  $\{\tilde{T}_i\}$ is a H\"older  complex associated with $\tilde{\mathcal{H}}$. Let $\{T_j\}$ be a refinement of $\{\tilde{T}_i\}$. Since $\tilde{T}_i$ are elementary triangles, the same is true for the triangles $T_j$. The structure of the pizza associated with the new triangulation can be obtained using the procedure described in Section 2. It is clear that $\mathcal{H}$ is a refinement of $\tilde{\mathcal{H}}$.

\begin{theorem} A minimal pizza associated with the germ of a definable continuous function $f$ is unique
up to combinatorial equivalence.
\end{theorem}

\begin{proof} Let $\tilde{\mathcal{H}}_1$ and $\tilde{\mathcal{H}}_2$ be two minimal pizzas corresponding to the same germ of a definable continuous function $f$. Let $\{\tilde{T}_{1,i}\}$ and $\{\tilde{T}_{2,k}\}$ be two H\"older complexes associated with $\tilde{\mathcal{H}}_1$ and $\tilde{\mathcal{H}}_2$, respectively. Consider a new H\"older complex $\{T_s\}$ obtained as a common refinement of $\{\tilde{T}_{1,i}\}$ and $\{\tilde{T}_{2,k}\}$. Using the geometric refinement procedure, one can construct a pizza $\mathcal{H}$ corresponding to the triangulation $\{T_s\}$. Then, the pizzas $\tilde{\mathcal{H}}_1$ and $\tilde{\mathcal{H}}_2$ are simplifications of the same pizza. Since the combinatorial equivalence class of a minimal pizza does not depend on the order of simplification operations, $\tilde{\mathcal{H}}_1$ and $\tilde{\mathcal{H}}_2$ are combinatorially equivalent.
\end{proof}

\begin{theorem} Two definable function germs $f,g\colon (\R^2,0)\rightarrow\R$ are contact Lipschitz
equivalent if, and only if, their minimal pizzas are combinatorially equivalent.
\end{theorem}

\begin{proof} If $f$ and $g$ are contact Lipschitz equivalent, then by Theorem \ref{theorem2} and
Lemma \ref{lemma4},  a minimal pizza of $f$ is a minimal pizza of $g$. Indeed if the pizza of $g$
(obtained from  Theorem \ref{theorem2}) were not minimal, any simplification would also result in a
simplification of the minimal pizza of $f$, which contradicts the definition. Thus respective minimal pizzas
of $f$ and $g$ are combinatorially equivalent.

\smallskip
If the minimal pizza of $f$ is combinatorially equivalent to the minimal pizza of $g$, then there exists a definable
bi-Lipschitz map $h\colon (\R^2,0)\rightarrow (\R^2,0)$ transforming the triangulation $\{T_i\}$, associated with $f$,
to the triangulation $\{T'_i\}$, associated with $g$. Let $H\colon (\R^3,0)\rightarrow (\R^3,0)$ defined by
$H(x,y,z)=(h(x,y),z)$ if the signs $s_i$ of the minimal pizza of $f$ are the same as the signs $s'_i$ of
the minimal pizza of $g$, and $H(x,y,z):=(h(x,y),-z)$ if the signs are opposite. The mapping $H$ transforms
the graph of $f$ into the graph of a function $\tilde{f}$. We are going to show that ${\tilde{f}(x,y)}/{g(x,y)}$
is bounded away from zero and infinity on the set of points where the functions are not zero. Notice that, by the
construction of $H$, the zero-sets of $\tilde{f}$ and $g$ are the same.

Let us suppose that $\tilde{f}/g$ is unbounded or tends to zero. Since $\tilde{f}$
and $g$ are definable, there exists an arc $\gamma$ such that $\tilde{f}/g$
on $\gamma$ is unbounded or tends to zero. But, by construction of the map $H$, the width of the arc $\gamma$ with respect to the functions $\tilde{f}$ and $g$ is equal to $\rm{tord}(\gamma,\gamma_i)$, where $\gamma_i$ is the marked boundary arc of the simplex $T_i$ such $\gamma\subset T_i$. That is why $\rm{ord}_{\gamma}(\tilde{f})=\rm{ord}_{\gamma}(g)$,  so that
$\tilde{f}/g$ is bounded below and above along $\gamma$.
This contradiction completes the proof.
\end{proof}

\section{Geometric realization of abstract pizzas}

Remind that we fixed a polynomially bounded o-minimal structure. In this section, we show that any abstract pizza can be realized as a geometric pizza associated with the germ of a definable function $f\colon(\R^2,0)\rightarrow (\R,0)$.

\begin{lemma}\label{connect} Let $T$ be a germ at the origin of a definable H\"older triangle with sides $\gamma_1$ and $\gamma_2$.
Let $f_1, f_2\colon  T,0\rightarrow \R,0$ be two nonnegative definable continuous functions such that $\rm{ord}_{\gamma_1}(f_1)=\rm{ord}_{\gamma_2}(f_2)=q\ge 0$.
There exists a nonnegative, definable, continuous in $T\setminus\{0\}$ function $f$  such that its restriction to $\gamma_i$ coincides with $f_i$, for $i=1,2$.
Moreover, $\rm{ord}_{\gamma}(f)=q$ for any arc $\gamma\subset T$.
If $q>0$ then the limit of $f$ at the origin is zero, thus $f$ is continuous in $T$.
If $q=0$ then $f$ is a unit, i.e., its values in $T\setminus\{0\}$ are separated from $0$ and $\infty$.
\end{lemma}

\begin{proof}
We may assume, by a definable bi-Lipschitz transformation,
that $T$ is bounded by the positive $x$-axis and a curve $y=\gamma(x)$
where $\gamma(x)\sim x^\beta$ for small $x>0$.
We define $f$ by linear interpolation:
$f(x,y)=(1-s)f_1(x,0)+s f_2(x,\gamma(x))$
where $s=y/\gamma(x)\in[0,1]$.
One can easily check that this function satisfies conditions of Lemma \ref{connect}.
\end{proof}

\begin{lemma}\label{explicit} Let $\mu\colon [q_1,q_2]\to\F\cap [1,\infty]$ be an affine function with the image $[\beta,\tilde{\beta}]$, where $0<q_1\le q_2$, $\beta\le\tilde\beta$ and $\beta<\tilde\beta\Leftrightarrow q_1<q_2$.
Let $T$ be a definable $\beta$-H\"older triangle.
Then, there exists a definable continuous function $f\colon (T,0)\rightarrow(\R,0)$ such that $\mu$ is the width function of $T$ associated with $f$.
\end{lemma}

\begin{proof} We may assume, by a definable bi-Lipschitz transformation,
that $T$ is bounded by the positive $x$-axis and a curve $y=\gamma(x)$
where $\gamma(x)\sim x^\beta$ for small $x>0$.

If $q_1=q_2$ and $\beta=\tilde\beta$, we define $f(x,y)=a(x)$
where $a(x)$ is any definable function such that $a(x)\sim x^{q_1}$ for small $x>0$.

Suppose now that $q_1<q_2$ and $\beta<\tilde\beta$.
Note that $\tilde\beta=\infty\Leftrightarrow q_2=\infty$.
In that case, $f(x,y)=b(y)a(x)$ where $\lambda={\rm d}\mu/{\rm d}q$, $b(y)\sim y^\lambda$, for small $y>0$,
and $a(x)\sim x^{q_1-\lambda\beta}$ for small $x>0$.

If $q_1<q_2<\infty$ and $\beta<\tilde\beta<\infty$,
let $\rho(x)$ be a definable function such that $\rho(x)\sim x^{\tilde\beta}$
for small $x>0$.
If ${\rm d}\mu/{\rm d}q>0$ then $f(x,y)=b(y+\rho(x))a(x)$
where $\lambda=(q_2-q_1)/(\tilde\beta-\beta)$, $b(z)\sim z^\lambda$, for small $z>0$, and
$a(x)\sim x^r$ with $r=(q_1\tilde\beta-q_2\beta)/(\tilde\beta-\beta)$
satisfies conditions of Lemma \ref{explicit}.

If ${\rm d}\mu/{\rm d}q<0$ then $f(x,y)=b(y+\rho(x)) a(x)$
where $\lambda=(q_1-q_2)/(\tilde\beta-\beta)$, $b(z)\sim z^\lambda$, for small $z>0$ and
$a(x)\sim x^r$ with $r=(q_2\tilde\beta-q_1\beta)/(\tilde\beta-\beta)$
satisfies conditions of Lemma \ref{explicit}.
\end{proof}

\begin{theorem}[Theorem of Realization] Let $\mathcal{H}=\{\beta_i,Q_i,s_i,\mu_i\}_{i\in I}$ be an abstract pizza
(see Definition \ref{abstract}). Then, there exists the germ of a definable continuous function $\psi\colon (\R^2,0)\rightarrow (\R,0)$ such that $\mathcal{H}$ is associated with $\psi$.
\end{theorem}

\begin{proof} Consider the H\"older complex associated with $\mathcal{H}$ and let us realize it in $\R^2$ at $0$.
Let $\{T_i\}_i$ be the set of the H\"older triangles from that realization, and
let $\gamma_i$ be the common boundary of $T_i$ and $T_{i+1}$.
Lemma \ref{explicit} allows us to construct definable continuous functions $f_i$ on each $T_i$ such that
$\mu_i:Q_i=[a_i,b_i]\to\F\cap\{\infty\}$ is the width function associated with $f_i$ on $T_i$,
the sign of $f_i$ inside $T_i$ is $s_i$,
${\rm ord}_{\gamma_{i-1}}(f_i)=a_i$ and ${\rm ord}_{\gamma_i}(f_i)=b_i$.
If $b_i=a_{i+1}<\infty$ then the function $\rho_i=(f_{i+1}|_{\gamma_i})/(f_i|_{\gamma_i})$
is positive and ${\rm ord}_{\gamma_i}(\rho_i)=0$.
Lemma \ref{connect} implies that there exists a continuous in $T_i\setminus\{0\}$ definable function
$g_i$ such that ${\rm ord}_\gamma(g_i)=0$ on each arc $\gamma$ in $T_i$, such that
$g_i|_{\gamma_{i-1}}\equiv 1$ and $g_i|_{\gamma_i}\equiv\rho_i$.
Then $\psi_i=f_i g_i$ is a continuous in $T_i$ definable function such that $\mu_i$
is the width function associated with $\psi_i$ in $T_i$, and $\psi_i|_{\gamma_i}\equiv \psi_{i+1}|_{\gamma_i}$.
Thus the function $\psi$ such that $\psi|_{T_i}=\psi_i$ is continuous, and $\mathcal{H}$ is
associated with $\psi$.
\end{proof}

\end{document}